\documentclass[a4paper, 11pt]{article}

\usepackage[utf8]{inputenc}
\usepackage[T1]{fontenc}
\usepackage{fullpage}

\usepackage{comment}
\usepackage{amsmath, amsthm, amssymb}
\usepackage{graphicx}
\usepackage{enumerate}
\usepackage{authblk}\usepackage{pgfplots}
\pgfplotsset{compat = newest}
\usepackage{thmtools}
\usepackage{thm-restate}
\usepackage[colorlinks=true, citecolor=red]{hyperref}
\usepackage{float}

\usepackage{tikz}
\renewenvironment{abstract}
{\small\vspace{-1em}
\begin{center}
\bfseries\abstractname\vspace{-.5em}\vspace{0pt}
\end{center}
\list{}{
\setlength{\leftmargin}{0.6in}\setlength{\rightmargin}{\leftmargin}}\item\relax}
{\endlist}
\usepackage{comment}
\declaretheorem[name=Theorem, numberwithin=section]{theorem}
\declaretheorem[name=Lemma, sibling=theorem]{lemma}

\declaretheorem[name=Definition, sibling=theorem]{definition}
\declaretheorem[name=Corollary, sibling=theorem]{corollary}
\declaretheorem[name=Conjecture, sibling=theorem]{conjecture}

\def\cqedsymbol{\ifmmode$\lrcorner$\else{\unskip\nobreak\hfil
\penalty50\hskip1em\null\nobreak\hfil$\lrcorner$
\parfillskip=0pt\finalhyphendemerits=0\endgraf}\fi}

    \def\R{\mathcal{R}}

\usetikzlibrary{calc}
\let\le\leqslant
\let\ge\geqslant
\let\leq\leqslant
\let\geq\geqslant

\newcounter{regle}
\newcommand{\regle}[1]{\refstepcounter{regle}\label{#1}R_{\theregle}}

\title{Improved square coloring of planar graphs\thanks{This work was supported by ANR project GrR (ANR-18-CE40-0032)}}

\author[2]{Nicolas Bousquet}
\author[2]{Quentin Deschamps}
\author[1,2]{Lucas de Meyer}
\author[2]{Théo Pierron}

\affil[1]{Département Informatique, ENS Rennes}
\affil[2]{Univ. Lyon, Université Lyon 1, LIRIS UMR CNRS 5205, F-69621, Lyon, France}

\date{}

\begin{document}

\maketitle

\begin{abstract}
    Square coloring is a variant of graph coloring where vertices within distance two must receive different colors. When considering planar graphs, the most famous conjecture (Wegner, 1977) states that $\frac32\Delta+1$ colors are sufficient to square color every planar graph of maximum degree $\Delta$. This conjecture has been proven asymptotically for graphs with large maximum degree. We consider here planar graphs with small maximum degree and show that $2\Delta+7$ colors are sufficient, which improves the best known bounds when $6\leqslant \Delta\leqslant 31$. 
\end{abstract}

\section{Introduction}

In graph theory, graph coloring is among one of the most studied problems. 
The history of graph coloring begins in the 19th century with the coloration of maps and the question asking whether four colors are sufficient to color differently regions sharing a same border.
This question can be rephrased as whether every planar graph can be properly colored with four colors, which was proved by Appel and Haken in 1976~\cite{appel1976every}. The proof of this result is known for being the first major computer-assisted proof.

Many other types of colorings were studied in the last decades, in particular in the case of planar graphs. We are interested here in the so-called \emph{square coloring}, where two vertices must receive distinct colors if they are at distance at most two. Given a graph $G$, the \emph{$2$-chromatic number} denoted by $\chi_2(G)$ is the minimum number of colors to color vertices at distance at most $2$ differently. The name ``square coloring'' comes from the fact that $\chi_2(G)$ can also be defined as the chromatic number of the square $G^2$ of $G$, \emph{i.e.} the graph obtained from $G$ by adding edges between vertices at distance two. 

Given a graph $G$, it is not hard to see that $\chi_2(G) \ge \Delta(G)+1$ where $\Delta(G)$ (or $\Delta$ when $G$ is clear from context) is the maximum degree of $G$. Indeed, all the neighbors of a vertex $v$ are at distance at most $2$ in $G$ and then must receive distinct colors. On the other hand, we have $\chi_2(G) \leq \Delta(G)^2+1$ since every vertex is adjacent to at most $\Delta(G)^2$ vertices in $G^2$ and in particular $G^2$ can be colored greedily with $\Delta(G)^2+1$ colors. One can prove that in general graphs this bound is tight for a finite number of graphs called the Moore graphs \cite{hoffman2003moore}, and asymptotically tight for the infinite family of incidence graphs of projective planes (which require $\Delta^2-\Delta+1$ colors).

In this paper, we focus on the particular case of planar graphs, that are graphs that can be embedded on the plane without edge-crossings. Planar graphs are known to be $5$-degenerate and thus we have $\chi_2(G) \le 5\Delta+1$. However, even if linear, this bound seems far from tight, since the graphs from Figure~\ref{fig:borneinf} satisfy $\chi_2(G)=\lfloor\frac{3\Delta(G)}{2}\rfloor+1$, which is the highest known value of $\chi_2$. 
\begin{figure}[!ht]
  \centering
  \begin{tikzpicture}[high/.style={inner sep=1.4pt, outer sep=0pt, circle, draw,fill=white}]
    \def\e{3}
    \def\es{1.2}
    \def\et{0.65}
    \def\R{2.1}
    \def\r{1.5}

    \begin{scope}[xshift=-2cm, yshift = 0cm, scale=0.8]
\draw

      (90:\e) node[high] (v1){}
      (90+120:\e) node[high] (v2){}
      (90+240:\e) node[high] (v3){}

      \foreach \i in {0,1,-1,2,-2}{
        (v2) ++ (60:\e/1.16)  ++ (60+90:0.3*\i) node[high](x){}
        (v2)--(x)-- (v1)
      }

      \foreach \i in {0,1,-1,2,-2}{
        (v1) ++ (-60:\e/1.16)  ++ (-60+90:0.3*\i) node[high](x){}
        (v1)--(x)-- (v3)
      }

      \foreach \i in {0,1,-1,2,-2}{
        (v3) ++ (-180:\e/1.16)  ++ (-180+90:0.3*\i) node[high](x){}
        (v2)--(x)-- (v3)
      }
      ;
      \draw[bend left] (v2)to(v3);
    \end{scope}
  \end{tikzpicture}
  \caption{Wegner's construction}
  \label{fig:borneinf}
\end{figure}
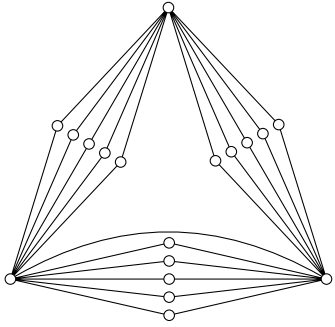

A famous conjecture from 1977 by Wegner~\cite{wegner1977graphs} states that up to small $\Delta$ cases, this should be the right upper bound. 
\begin{conjecture}[Wegner~\cite{wegner1977graphs}]
Every planar graph G with maximum degree $\Delta$ satisfies:
$$
\chi_2(G) \le \left\{
    \begin{array}{ll}
        7 & \mbox{if } \Delta = 3, \\
        \Delta + 5 & \mbox{if } 4 \le \Delta \le 7, \\
        \lfloor \frac{3\Delta}{2} \rfloor  + 1 & \mbox{if } \Delta \ge 8.
    \end{array}
\right.
$$
\end{conjecture}

Despite receiving considerable attention, Wegner's conjecture is still open today, except for the case of subcubic graphs solved by Thomassen~\cite{thomassen2001applications}. As shown independently by~\cite{havet2017list} and~\cite{amini2013unified}, the conjecture asymptotically holds: $\chi_2(G)=3/2 \Delta +o(\Delta)$ when $\Delta\to\infty$. Many results of the form $c\Delta+O(1)$ (where $c$ is a constant) were also found, culminating with $c=5/3$, obtained by Molloy and Salavatipour~\cite{molloy2005bound}.

However, since the constants hidden in these proofs are large (and these results only hold for large enough $\Delta$), the picture is far from being complete, especially for small values of $\Delta$. An extensive line of work consisted in improving function when $\Delta$ is small. A table of some of the existing results is summarized in Table~\ref{table}.

\begin{center}\label{table}
\begin{tabular}{|c|c|c|}
     \hline 
     Authors & Restriction & Result  \\
     \hline
     Thomassen \cite{thomassen2001applications}& $\Delta \leq 3$ & $\chi_2(G) \leq 7$ \\
     \hline
     Jonas \cite{Jonas1993graph} & $\Delta \geq 7$ &  $\chi_2(G) \leq 8 \Delta -22$ \\
    \hline
     Wong \cite{wong1996colouring} & $\Delta \geq 7$ &   $\chi_2(G) \leq 3 \Delta +5$ \\
     \hline
     Madaras and Marcinova \cite{madaras2002structural} & $\Delta \geq 12$ &   $\chi_2(G) \leq 2 \Delta +18$ \\
     \hline
     & $\Delta \leq 20$ &   $\chi_2(G) \leq 59$ \\
     Borodin et al. \cite{borodin2002stars} & $21 \leq \Delta \leq 46$ &   $\chi_2(G) \leq \Delta + 39$ \\
     & $\Delta \geq 47$ &   $\chi_2(G) \leq \lceil \frac{9 \Delta}{5} \rceil +1$ \\
     \hline
     Van den Heuvel and McGuinness \cite{van2003coloring} & $\Delta \geq 5$ &$\chi_2(G) \leq 9\Delta -19$ \\
     & & $\chi_2(G) \leq 2\Delta +25$ \\
     \hline
    Agnarsson and Halldorsson \cite{agnarsson2003coloring} & $\Delta \geq 749$ &   $\chi_2(G) \leq \lfloor \frac{9 \Delta}{5} \rfloor +2$  \\
     \hline
    Molloy and Salavatipour \cite{molloy2005bound} & $\Delta \geq 249$ &$\chi_2(G) \leq \lceil \frac{5 \Delta}{3} \rceil +25$ \\
    & & $\chi_2(G) \leq \frac{5 \Delta}{3} \rceil +78$ \\
    \hline
    Zhu and Bu \cite{zhu2018minimum} & $\Delta \leq 5$ &$\chi_2(G) \leq 20$ \\
    & $\Delta \geq 6$ & $\chi_2(G) \leq 5 \Delta -7$ \\
    \hline
    Krzyzinski et al. \cite{krzyzinski2021coloring} & $\Delta \ge 6$ & $\chi_2(G) \le 3 \Delta +4$\\
    \hline
\end{tabular}
\end{center} 
When we represent all these results on the same graph depending on $\Delta$, we obtain the graph depicted in Figure~\ref{fig:courbe}. We follow this line of research by improving the best known bound on $\chi_2(G)$ for any planar graph whose maximum degree is between $6$ and $31$. More formally, we prove the following.

\begin{theorem}\label{thm:main}
Let $G$ be a planar graph of maximal degree $\Delta \geq 9$. Then, $\chi_2(G) \leq 2\Delta +7$.
\end{theorem}

For $\Delta\leqslant 6$, the $2\Delta+7$ bound is 19. Reusing our analysis, we can easily prove a slightly worse bound, which still improves the best current bound of~\cite{krzyzinski2021coloring}.  
\begin{theorem}
\label{thm:delta6}
Let $G$ be a planar graph of maximal degree $\Delta \leq 6$. Then, $\chi_2(G) \leq 21$.
\end{theorem}

\begin{figure}[hbtp]
    \begin{center}

    \pgfplotstableread{conjecture.dat}{\table}

        \begin{tikzpicture}[scale=0.7]
        \begin{axis}[
            xmin = 2, xmax = 26,
            ymin = 5, ymax = 70,
            xtick distance = 1,
            ytick distance = 4,
            grid = both,
            minor tick num = 0,
            major grid style = {lightgray!50},
            minor grid style = {lightgray!25},
            width = 1.43\textwidth,
            height = 0.95\textwidth,
            xlabel = {Maximum degree $\Delta$},
            ylabel = {Bound on $\chi_2(G)$},
            legend cell align = {left},
            legend pos = north west,
            grid style = dashed,
        ]
        
        \addplot[blue , mark = *] {0};
        \addplot[red , mark = *] {0};
        \addplot[purple , mark = *] {0};
        \addplot[yellow , mark = *] {0};
        \addplot[green , mark = *] {0};
        \addplot[orange , mark = *] {0};
        \addplot[teal, mark = *, mark size = 2.7pt] {0};
        
        \addlegendentry{Conjecture};
        \addlegendentry{Thomassen};
        \addlegendentry{Zhu and Bu};
        \addlegendentry{Rzążewski et al.};
        \addlegendentry{Madaras and Marcinová};
        \addlegendentry{Borodin et al. };
        \addlegendentry{Theorems~\ref{thm:main} and~\ref{thm:delta6}}
        
        \addplot[blue, mark = *] table [x = {x}, y = {y1}] {\table};
        
         \addplot[red!40, mark = *, dashed, domain = 2:3, samples = 2] {7} ;
        \addplot[red, only marks = *]coordinates {
             (3,7) };
            
        \addplot[purple!40, mark = *, dashed, domain = 2:5, samples = 4] {20} ;
        \addplot[purple!40, mark = *, dashed,] coordinates {(5,20) (6,23) } ;
        \addplot[purple!40, mark = *, dashed, domain = 6:26, samples = 21] {5*x - 7} ;
        \addplot[purple, mark = *, domain = 4:5, samples = 2] {20} ;
        
        \addplot[yellow!40, mark = *, dashed, domain = 2:6, samples = 5] {22} ;
        \addplot[yellow!40, mark = *, dashed, domain = 6:26, samples = 21] {3*x + 4} ;
        \addplot[yellow, mark = *, domain = 6:14, samples = 9] {3*x + 4} ;

        \addplot[green!40, dashed, mark = *, domain= 2:12, samples = 11] {42};
        \addplot[green!40, dashed, mark = *, domain= 12:26, samples = 15] {2*x + 18};
        \addplot[green, mark = *, domain= 14:21, samples = 8] {2*x + 18};
        
        \addplot[orange!40, dashed, mark = *, domain= 2:20, samples = 19] {59};
        \addplot[orange!40, dashed, mark = *, domain= 20:24, samples = 5] {x + 39};
        \addplot[orange, mark = *, domain= 21:26, samples = 6] {x + 39};

        \addplot[teal!40, dashed, mark = *, mark size = 2.7pt, domain = 2:6, samples = 5] {21};
        \addplot[teal, mark = *, mark size = 2.7pt, domain = 9:26, samples = 18] {2*x + 7};
        \addplot[teal, mark = *, mark size = 2.7pt, domain = 7:9, samples = 3] {25};
        \addplot[teal, mark = *, mark size = 2.7pt, domain = 6:7, samples = 2] {4*x-3};
        \end{axis}
        \end{tikzpicture}

    \end{center}
    \caption{Bound on $\chi_2(G)$ depending to $\Delta$ for $3\le \Delta \le 26$}
    \label{fig:courbe}
\end{figure}
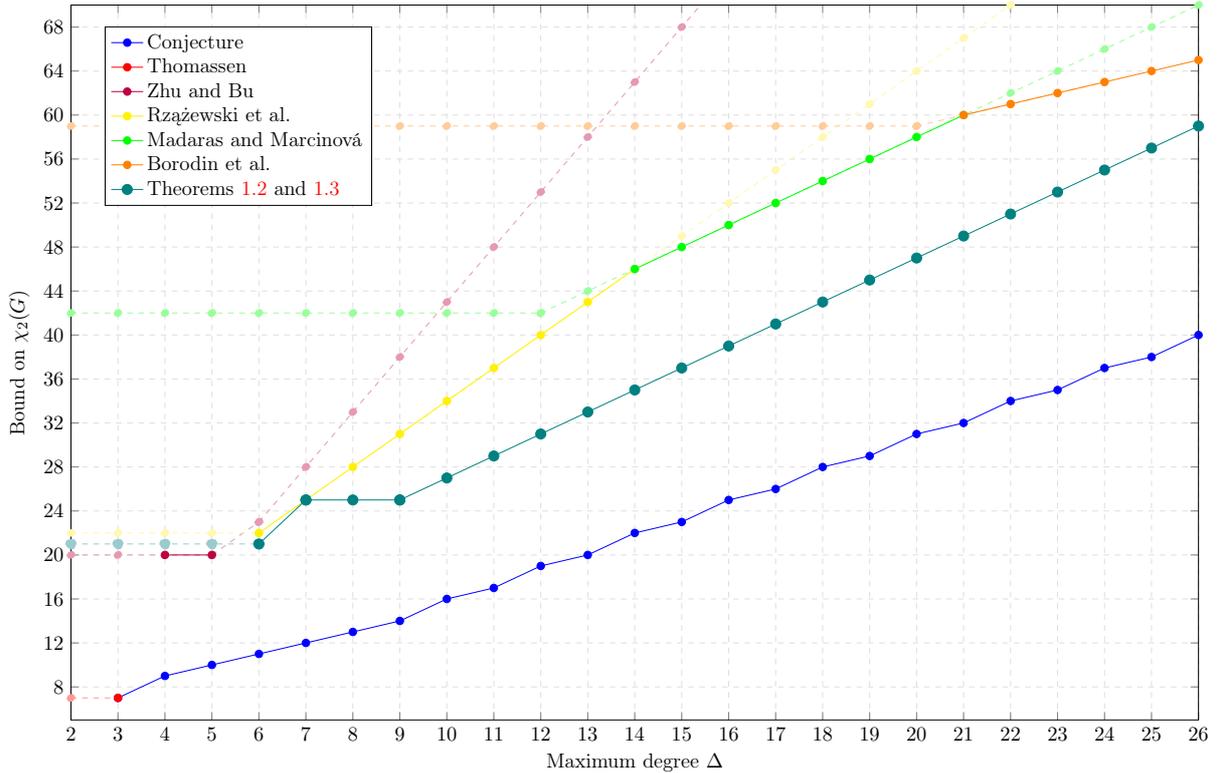
     
\section{Preliminaries}

Given a graph $G = (V, E)$, we denote by respectively $V(G)$ and $E(G)$ the \emph{vertex set} and the \emph{edge set} of $G$. Since our goal is to color $G$, we assume that it does not contain loops nor parallel edges. All along the paper, we will assume that $G$ is \emph{planar}, namely vertices and edges can be drawn in the plane so that no two edges cross (except maybe on their common extremities). When such an embedding is fixed, we denote by $F(G)$ the set of the faces of $G$.

Let $v \in V(G)$. We denote by $N_G(v)$ the \emph{neighborhood} of $v$ in $G$, i.e., the set of the vertices adjacent to $v$. We define  the \emph{degree}  $\deg_G(v)$ of $v$ in $G$ as the size of $N_G(v)$. The maximum and minimum degrees of $G$ are respectively denoted by $\Delta(G)$ and $\delta(G)$. Finally, the \emph{degree of a face} $f$, denoted by $\deg_G(f)$, is the number of edges incident to $f$, where each incident cut-edge is counted twice.

For an integer $d$, a vertex is said to be a \emph{$d$-vertex} (respectively, a \emph{$d^-$-vertex}, a \emph{$d^+$-vertex}) if its degree is equal to $d$ (resp., at most $d$, at least $d$). Likewise, a face is said to be a \emph{$d$-face} (resp., a \emph{$d^-$-face}, a \emph{$d^+$-face}) if its degree is equal to $d$ (resp., at most $d$, at least $d$).

Let $X \subseteq V(G)$. We denote by $G[X]$ the subgraph induced by the set of vertices $X$.
The graph obtained from $G$ by removing a vertex $v$ and its incident edges is denoted by $G-v$.
Note that $G - v = G [V \setminus v]$.
Let $u, v$ in $V$.
By $G + uv$, we denote the graph $G'$ such that $V(G') = V(G)$ and $E(G') = E(G) \cup \{uv\}$, i.e., the graph obtained from $G$ by adding an edge between $u$ and $v$ if it does not already exist.

A \emph{partial coloring} on a set of vertices $X$ of a graph $G$ is a coloring for the subgraph induced by $X$ in $G$ such that every pair of vertices at distance $2$ in $G$ are colored differently. 

\section{Graph ordering}

To prove Theorem~\ref{thm:main}, we assume that there exists a minimal counterexample and exhibit a contradiction. We use a specific order on graphs which is the following.

\begin{definition}
Let $G_1=(V_1,E_1)$ and $G_2=(V_2,E_2)$ be two graphs. We say $G_1 \leq G_2$ when $|V_1|<|V_2|$ or $|V_1|=|V_2|$ and $|E_1| \geq |E_2|$. In other words $\leq$ is the lexicographic order on $(|V|,-|E|)$.
\end{definition}

In the following, we take a counterexample $G$ to Theorem~\ref{thm:main} minimal for $\leq$ and we fix an arbitrary planar embedding of $G$.

\begin{lemma}\label{lem:cutedgeforbif}
The graph $G$ does not contain an edge-separator i.e. an edge $uv$ such that $G\setminus \{u,v\}$ is disconnected.
\end{lemma}
\begin{proof}
     Assume by contradiction that $G$ contains an edge-separator $uv$. Let $C$ be a connected component of $G \setminus \{u,v\}$. Let $G_1 = G[C \cup \{u,v\}]$ and $G_2 = G \setminus V(C)$.
     The graphs $G_1$ and $G_2$ have strictly fewer vertices than $G$. So, by minimality of $G$, there exist $\alpha_1$ and $\alpha_2$ two distance-2 $(2\Delta+7)$-colorings of respectively $G_1$ and $G_2$. Up to color permutation, we can assume that $u$ and $v$ are colored respectively $1$ and $2$ on both colorings $\alpha_1$ and $\alpha_2$. For $1 \le i \le 2$, let $N_i=(N_{G_i}(u) \cup N_{G_i}(v)) \setminus \{u, v\}$. As $G$ has maximum degree $\Delta$, $|N_1| + |N_2| \leq 2 \Delta$. So we can permute the colors in $G_1$ and $G_2$, except the color $1$ and $2$, in such a way the vertices of $N_1$ are colored differently from the vertices of $N_2$.
     Now the coloring of $G$ obtained by merging the colorings $\alpha_1$ and $\alpha_2$ is a distance-$2$ coloring for $G$.
\end{proof}
An easy consequence of our choice of ordering yields the following. 

\begin{theorem}\label{thm:edge}
Every $4^+$-face of $G$ contains at most $2$ vertices of degree less than $\Delta$. Moreover, if a face $f$ contains exactly two such vertices, they are consecutive in the face $f$.
\end{theorem}

\begin{proof}
Assume by contradiction that $G$ contains a $4^+$-face $f$ with $2$ non-adjacent vertices $u,v$ of degree less than $\Delta$. We claim that $u$ and $v$ cannot be connected via an edge not in $f$ since otherwise this edge would be an edge-separator, contradicting Lemma~\ref{lem:cutedgeforbif}. Therefore, adding the edge $uv$ to $G$ yields a (planar) graph $G'$ with $G' < G$ and $\Delta(G')=\Delta(G)$. Since $G$ is a subgraph of $G'$, $\chi_2(G) \leq \chi_2(G') \leq 2\Delta + 7$ by minimality, a contradiction.

Now if $f$ contains three vertices of degree less than $\Delta$, two of them are non adjacent since $f$ is a $4^+$-face, a contradiction.
\end{proof}

\section{Forbidden configurations}

We will use the same method to prove that our minimal counterexample $G$ does not contain some configurations. Namely, we first assume by contradiction that the configuration is in $G$.
Then, we remove a vertex $v$ and possibly add edges to obtain a smaller planar graph $G'$ such that $\Delta(G')\leqslant \Delta(G)$, and every pair of vertices at distance at most $2$ in $G$ are still at distance $2$ in $G'$.

Since $G'$ is smaller than $G$, $G'$ admits a distance-$2$ $(2\Delta + 7)$-coloring $\alpha$,  and $\alpha$ is a partial coloring of $G$ since the adjacencies at distance at most $2$ are preserved.

In order to extend $\alpha$ to $V(G)$, it is sufficient to color $v$. If there are less than $2\Delta + 7$ forbidden colors for $v$, then we can color $v$ with a free color. This provides a distance-2 coloring of $G$, a contradiction. This proves the following.

\begin{lemma}\label{lem:minconstraint}
The graph $G$ does not contain a vertex $v$ such that, in $G-v$, all the pairs of neighbors of $v$ are at distance at most $2$ and such that $v$ has less than $2 \Delta +7$ neighbors at distance $2$.
\end{lemma}
\begin{comment}
\begin{proof}
    By contradiction, let $\alpha$ be a distance-$2$ coloring of $G' = G-v$ with $2 \Delta +7$ colors. Let us extend $\alpha$ to $G$ by setting $\alpha(v)$ to a free color not used by any $2$-neighbor of $v$ (which exists by hypothesis). Then $\alpha$ is a distance-$2$ coloring of $G$, a contradiction. 
\end{proof}
\end{comment}

In particular, we get that $\delta(G)\geq 2$.
\begin{corollary}\label{lem:deg1forbif}
The graph $G$ does not contain any $1$-vertex. 
\end{corollary}

For brevity, in what follows, we only give the graph $G'$. Thus, we will not explicitly check that $G'<G$, $\Delta(G')\leq \Delta(G)$, and that the adjacencies at distance 2 are preserved. We only compute the number of forbidden colors for $v$ and check that it is smaller than $2\Delta + 7$. 

We may now extend the previous result to $\delta(G)\geqslant 3$, and show a number of other forbidden configurations.

\begin{lemma}\label{lem:deg2forbif}
The graph $G$ does not contain any $2$-vertex. 
\end{lemma}
\begin{proof}
    By contradiction, let $v$ be a $2$-vertex of $G$. Let $u$ and $w$ be the two neighbors of $v$ in $G$ and $G' = G-v + uw$. By minimality of $G$, $G'$ admits a distance-$2$ $(2\Delta+ 7)$-coloring $\alpha$. The coloring $\alpha$ is a partial coloring of $G$ and forbids at most $2\Delta$ colors for $v$. We can extend $\alpha$ to $V(G)$, a contradiction.
\end{proof}

\begin{lemma}\label{lem:deg3forbif4}
The graph $G$ does not contain a $3$-vertex adjacent to a $5^-$-vertex. 
\end{lemma}
\begin{proof}
    Assume by contradiction that $G$ has a $3$-vertex $v$ adjacent to a $5^-$-vertex $u$ such that $N(v) = \{u, v_1, v_2\}$. 
    Let $G'$ be the graph obtained by identifying $u$ and $v$, i.e. $G' = G - v + uv_1 + uv_2$. 
    Note that $\deg_{G'}(u) = \deg_G(u)+1 \le 6 \le \Delta$.
    By minimality, $G'$ admits a distance-$2$ $(2\Delta + 7)$-coloring $\alpha$.
    The coloring $\alpha$ is a partial coloring of $G$ and forbids at most $2\Delta + 5$ colors for $v$.
    Therefore, we can extend the coloring $\alpha$ to $v$, a contradiction.
\end{proof}

\begin{lemma}\label{lem:deg3forbif}
The graph $G$ does not contain a degree-$3$ vertex incident to two $3$-faces and adjacent to a vertex of degree at most $\min(10,\Delta)$ (see Figure \ref{fig:deg3bis} for an illustration).
\end{lemma}
 \begin{figure}[hbtp]
        \begin{center}
        \tikzstyle{vertex}=[circle,draw, minimum size=15pt, scale=1, inner sep=1pt]
        \tikzstyle{give}=[circle,draw, minimum size=15pt, scale=1, inner sep=1pt, fill=black!5, very thick]
        \tikzstyle{receive}=[circle,draw, minimum size=15pt, scale=1, inner sep=1pt, fill=black!5]
        \tikzstyle{fleche}=[->,>=latex]
        \begin{tikzpicture}[scale=1]
            
        \node (a0) at (0,0) [vertex,label=left:{$v$}] {3};
        \node (b0) at (1,1) [vertex] {$v_1$};
        \node (c0) at (1,-1) [vertex] {$v_3$};
        \node (d0) at (2,0) [vertex] {$v_2$};

        \draw (a0) to (b0);
        \draw (a0) to (c0);
        \draw (a0) to (d0);
        \draw (c0) to (d0);
        \draw (d0) to (b0);

        \end{tikzpicture}
        \end{center}
        \caption{Forbidden configuration of Lemma~\ref{lem:deg3forbif} where $\deg(v_i) \le \min(10, \Delta)$ for some $i \in \{1,2,3\}$}
        \label{fig:deg3bis}
    \end{figure}
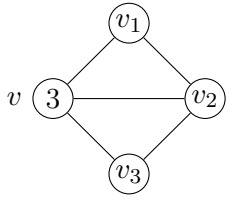
     \begin{proof}
    Assume by contradiction that $G$ has a $3$-vertex $v$ such that $N(v)=\{ v_1, v_2, v_3 \}$ and $v_2$ is adjacent to $v_1$ and $v_3$. 
    Let $G' = G - v$ and $d = \min(10,\Delta)$.
    Since $G$ is a minimal counterexample, $G'$ admits a distance-$2$ $(2\Delta + 7)$-coloring $\alpha$.
    The partial coloring $\alpha$ forbids at most $(2\Delta + d) - 7 + 3$ colors for $v$.
    Note that since $d < 11$, we have $2\Delta + d  - 4 < 2\Delta + 7$.
    So we can extend $\alpha$ to $v$, a contradiction.
\end{proof}

\begin{lemma}\label{lem:deg3forbif2}
The graph $G$ does not contain a $3$-vertex incident to a $3$-face and two $4$-faces when $\Delta \le 10$ (see Figure \ref{fig:deg3square}).
\end{lemma}
 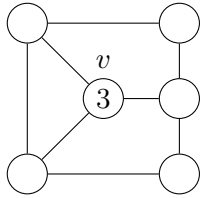
\begin{figure}[hbtp]
        \begin{center}
        \tikzstyle{vertex}=[circle,draw, minimum size=15pt, scale=1, inner sep=1pt]
        \tikzstyle{vertex1}=[circle,draw, minimum size=15pt, scale=1, inner sep=1pt]
        \tikzstyle{give}=[circle,draw, minimum size=15pt, scale=1, inner sep=1pt, fill=black!5, very thick]
        \tikzstyle{receive}=[circle,draw, minimum size=15pt, scale=1, inner sep=1pt, fill=black!5]
        \tikzstyle{fleche}=[->,>=latex]
        \begin{tikzpicture}[scale=1]
            
        \node (a0) at (1,0) [vertex1,label=above:{$v$}] {3};
        \node (b0) at (0,1) [vertex] {};
        \node (c0) at (0,-1) [vertex] {};
        \node (d0) at (2,0) [vertex] {};
        \node (e0) at (2,1) [vertex] {};
        \node (f0) at (2,-1) [vertex] {};
        
        \draw (a0) to (b0);
        \draw (a0) to (c0);
        \draw (a0) to (d0);
        \draw (c0) to (f0);
        \draw (d0) to (e0);
        \draw (d0) to (f0);
        \draw (b0) to (e0);
        \draw (c0) to (b0);

        \end{tikzpicture}
        \end{center}
        \caption{Forbidden configuration of Lemma~\ref{lem:deg3forbif2}}.
        \label{fig:deg3square}
    \end{figure} 
\begin{proof}
    Assume by contradiction that $G$ has a such $3$-vertex $v$. 
    Let $G' = G - v$.
    Since $G$ is a minimal counterexample, $G'$ admits a distance-$2$ $(2\Delta + 7)$-coloring $\alpha$. 
    The partial coloring $\alpha$ forbids at most $3(\Delta - 2) + 2< 2\Delta + 7$ colors for $v$. Hence we can extend $\alpha$ to $v$, a contradiction.
\end{proof}

A vertex is \emph{triangulated} if it is only incident to $3$-faces.

\begin{lemma}\label{lem:deg4forbif3}
The graph $G$ does not contain a triangulated $4$-vertex adjacent to a triangulated $5$-vertex and to a vertex of degree less than $12$ (see Figure~\ref{fig:deg4tri5}). 
\end{lemma}
    \begin{figure}[hbtp]
        \begin{center}
        \tikzstyle{vertex}=[circle,draw, minimum size=10pt, scale=1, inner sep=1pt]
        \tikzstyle{give}=[circle,draw, minimum size=15pt, scale=1, inner sep=1pt, fill=black!5, very thick]
        \tikzstyle{receive}=[circle,draw, minimum size=15pt, scale=1, inner sep=1pt]
        \tikzstyle{fleche}=[->,>=latex]
        \begin{tikzpicture}[scale=1]

            \node (a0) at (0,0) [receive] {4};
            \node at (-.3,.3) {$v$};
            \node (b0) at (1,0) [receive,label=right:{$w$}]{5};
            \node (c0) at (-1,0) [vertex] {};
            \node (d0) at (0,1) [vertex]{};
            \node (e0) at (0,-1) [vertex]{};
            \node (f0) at (2,1) [vertex]{};
            \node (g0) at (2,-1) [vertex]{};

            \draw (b0) to (d0);
            \draw (c0) to (d0);
            \draw (b0) to (e0);
            \draw (c0) to (e0);
            \draw (a0) to (b0);
            \draw (a0) to (c0);
            \draw (a0) to (d0);
            \draw (a0) to (e0);
            \draw (b0) to (f0);
            \draw (g0) to (f0);
            \draw (b0) to (g0);
            \draw (d0) to (f0);
            \draw (g0) to (e0);

        \end{tikzpicture}
        \end{center}
        \caption{Forbidden configuration of Lemma~\ref{lem:deg4forbif3} with a neighbor $u$ of degree at most $12$}
        \label{fig:deg4tri5}
    \end{figure}
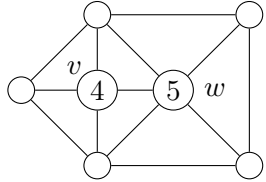 
\begin{proof}
    Assume by contradiction that $G$ has a triangulated $4$-vertex $v$ such that $v$ is adjacent to a triangulated $5$-vertex $w$ and a vertex of degree at most $d < 12$. 
    Let $G' = G - v$. 
    By minimality, $G'$ admits a distance-$2$ $(2\Delta + 7)$-coloring $\alpha$. 
    The partial coloring $\alpha$ forbids for $v$ at most $(2\Delta + d) - 2\cdot3 - 2 + 3 = 2\Delta + d - 5 < 2d + 7$ colors when $d < 12$. 
    So we can extend $\alpha$ to $v$, a contradiction.
\end{proof}

\section{Discharging rules}

We give an initial charge $\omega_0(v) = \deg(v) - 6$ to every vertex $v$ and an initial charge $\omega_0(f) = 2\deg(f) - 6$ to every face $f$.
A well-known consequence of Euler's formula is the following equality:
$$ \sum_{v \in V(G)}{(\deg(v) - 6)} + \sum_{f \in F(G)}{(2\deg(f) - 6)} =- 12 < 0 \ \ (\star) $$
The charge will be transferred via $4$ discharging rules in such a way that the final charge of every face and every vertex will be non-negative, which contradicts $(\star)$ and proves Theorem~\ref{thm:main}.

The discharging procedure is in two steps. We first apply the following discharging rule :

\begin{itemize}
    \item[($\regle{rface}$)] A $d$-face gives $d-3$ to each incident $5^+$-vertex. 
\end{itemize}

Before going to the second step and defining the other discharging rules, we give some definitions. 
A \emph{weak} vertex is a vertex $v$ with a negative charge after applying $R_{\ref{rface}}$.
A \emph{strong} vertex is a vertex which is not weak.  
Note that every weak vertex has degree at most $5$ and then every $6^+$-vertex is strong.

A vertex $u$ is \emph{next to a vertex $w$ around a vertex $v$} if $u$ and $w$ are consecutive in the cyclic ordering of the neighbors of $v$ in the plane embedding. 
Note that a vertex $u$ is next to at most two other vertices around a fixed vertex $v$.

To simplify the upcoming analysis, some rules make charge transit through some intermediate vertices. Note that this only affects the final charge of the source and target vertices.

We use these definitions to define the other discharging rules (see Figure~\ref{fig:rulesbis2}) :

\begin{itemize}
    \item[($\regle{r11}$)] A $11$-vertex gives $\frac{5}{11}$ to each weak neighbor if all its neighbors are weak. Otherwise, $R_{\ref{r11+}}$ applies.
    
    \item[($\regle{r11+}$)] A $11^+$-vertex gives $\frac{1}{2}$ to each weak neighbor $w$. Moreover, if there are two strong vertices next to $w$, an additional $\frac{1}{4}$ transits through them before reaching $w$. Therefore, the total charge moved is either $\frac 12$ or $1$.

    \item[($\regle{r7}$)] A vertex $v$ of degree $7 \le d \le 10$ gives $\frac{\omega_0}{d}$ to each weak neighbor $w$, plus $\frac{\omega_0}{2d}$ transiting through each strong vertex next to $w$. So $w$ can receive $\frac{\omega_0}{d}$, $\frac{3\omega_0}{2d}$ or $\frac{2\omega_0}{d}$ from $v$.
\end{itemize}

 \begin{figure}[hbtp]
        \begin{center}
        \tikzstyle{vertex}=[circle,draw, minimum size=15pt, scale=1, inner sep=1pt]
        \tikzstyle{give}=[circle,draw, minimum size=15pt, scale=1, inner sep=1pt, fill=black!5, very thick]
        \tikzstyle{receive}=[circle,draw, minimum size=15pt, scale=1, inner sep=1pt, fill=black!5]
        \tikzstyle{fleche}=[->,>=latex]
        \begin{tikzpicture}[scale=1]
            \node [draw] at(-1,0){$R_{\ref{rface}}$};
            \node (a0) at (2,0) [give] {$F$};
            \node (b0) at (4,0) [receive] {$5^-$};
            \draw [fleche](a0) to  node[pos=0.5,fill=white] {$d-3$} (b0); 
            \node (c0) at (0,0) [vertex]{};
            \node (d0) at (2,-1) [vertex]{};
            \node (e0) at (2,1) [vertex]{};
            \draw (b0) to (d0);
            \draw (c0) to (d0);
            \draw (b0) to (e0);
            \draw (c0) to (e0);

        \tikzset{xshift=7cm}
            
            \node [draw] at(-1,0){$R_{\ref{r11}}$};
            \node (a4) at (0,0) [give] {$11$};
            \node (b4) at (4,0) [receive] {$w$};
            \draw [fleche](a4) to  node[pos=0.5,fill=white] {$\frac{5}{11}$} (b4);

        \tikzset{yshift=-3.5cm, xshift=-7cm}
        
        \node [draw] at(-1,0){$R_{\ref{r11+}}$};
            \node (a2) at (0,0) [give] {$11^+$};
            \node (b2) at (4,0) [receive] {$w$};
            \node (c2) at (2,-1.5) [vertex] {$s$};
            \node (d2) at (2,1.5) [vertex] {$s$};
            \draw [fleche](a2) to  node[pos=0.5,fill=white] {$\frac{1}{2}$} (b2); 
            \draw [fleche](a2) to  node[pos=0.5,fill=white] {$\frac{1}{4}$} (2,1) to (b2); 
            \draw [fleche](a2) to  node[pos=0.5,fill=white] {$\frac{1}{4}$} (2,-1) to (b2); 
            \draw (a2) to (c2);
            \draw (a2) to (d2);
            \draw [dashed] (b2) to (c2);
            \draw [dashed] (b2) to (d2);
            
\tikzset{xshift=7cm}
            
            \node [draw] at(-1,0){$R_{\ref{r7}}$};
            \node (a6) at (0,0) [give] {$7/10$};
            \node (b6) at (4,0) [receive] {$w$};
            \node (c6) at (2,1.5) [vertex] {$s$};

            \draw [fleche](a6) to  node[pos=1,fill=white,inner sep = 0pt] {$\frac{\omega_0}{2d}$} (2, .9) to (b6); 
            \draw [fleche](a6) to  node[pos=0.5,fill=white] {$\frac{\omega_0}{d}$} (b6); 
            \draw (a6) to (c6);
            \draw [dashed] (b6) to (c6);

        \end{tikzpicture}
        \end{center}
        \caption{Discharging rules. Bold nodes represent the origin of the weight. Labels $s$ and $w$ denote strong and weak vertices. Non-straight arrows denote the charge transiting through strong neighbors.}
        \label{fig:rulesbis2}
    \end{figure}
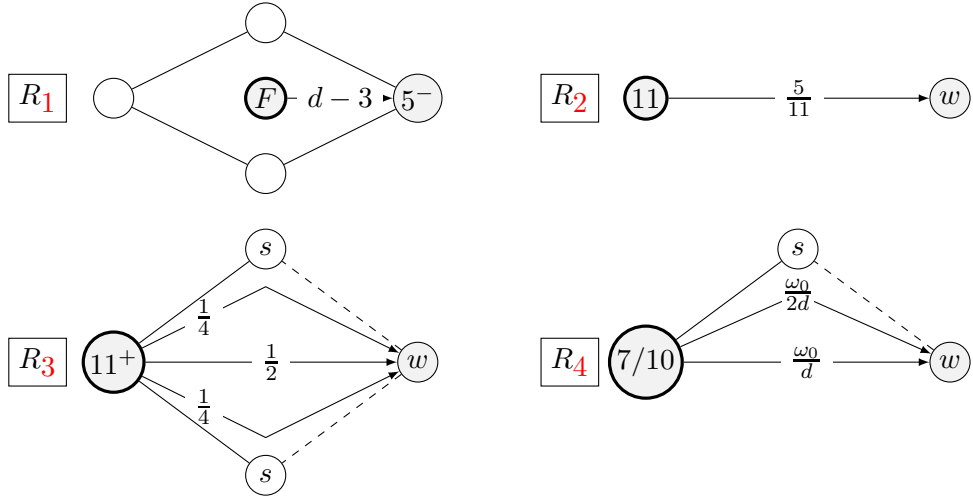

\section{Final charges}

We prove in this section that the final charges are non-negative. 
We separate the proof for the faces and then for vertices depending on their degree.

\begin{lemma}\label{claim:faces}
    The faces of $G$ have non-negative charge after applying the discharging rules.
\end{lemma}
\begin{proof}
    Let $d \geq 3$ and let $f$ be a $d$-face. The initial charge of $f$ is $2d - 6 \geq 0$.
    If $d = 3$, $f$ does not give nor receive any charge and its final charge is 0.
    
    Otherwise, $f$ is a $4^+$-face and it only gives charge because of $R_{\ref{rface}}$.     By Lemma~\ref{thm:edge}, $f$ does not contain two non-adjacent $5$-vertices, hence $f$ has at most two $5^-$-vertices on its boundary and gives at most $2\cdot (d - 3) = 2d - 6$.
    So the final charge of $f$ is non-negative.
\end{proof}

\begin{lemma}
\label{lem:ok7}
    Any vertex of degree at least 7 and at most 10 has a positive charge after applying the discharging rules.
\end{lemma}
\begin{proof}
    Let $v$ be a vertex of degree $d$, with $7 \le d \le 10$.
    The initial charge of $v$ is $\omega_0 = d - 6$.
    Note that only $R_{\ref{r7}}$ may apply.
    
    The vertex $v$ gives at most $\frac{\omega_0}{d}$ to adjacent weak vertices.
    Moreover, every strong neighbor of $v$ is next to at most two vertices around $v$, so at most $2\cdot\frac{\omega_0}{2d} = \frac{w_0}{d}$ charge transits through each strong vertex by $R_{\ref{r7}}$. The other neighbors of $v$ do not receive any charge from $v$. 
    
    Therefore, at most  $\frac{\omega_0}{d}$ leaves $v$ towards each of its $d$ neighbors, thus the final charge of $v$ is at least $\omega_0 - d \cdot \frac{\omega_0}{d} \ge 0$.
\end{proof}

\begin{lemma}
    Vertices of degree at least 12 have a positive charge after applying the discharging rules.
\end{lemma}
\begin{proof}
    Let $v$ be a vertex of degree $d \ge 12$. Its initial charge is $d - 6$.
    Note that only $R_{\ref{r11+}}$ may apply. The same analysis as in Lemma~\ref{lem:ok7} yields that at most $\frac{1}{2}$ charge leaves $v$ towards each of incident edges, thus its final charge is at least $d-6 - d \cdot \frac{1}{2} \ge 0$.
\end{proof}

\begin{lemma}
    Vertices of degree 11 have a positive charge after applying the discharging rules.
\end{lemma}
\begin{proof}
    Let $v$ be a vertex of degree $11$ which then has an initial charge of $5$.
    Note that only $R_{\ref{r11}}$ and $R_{\ref{r11+}}$ may apply. If no neighbor of $v$ is strong then $R_{\ref{r11}}$ applies and the final charge is $5-11\cdot \frac{5}{11}=0$. Therefore, we may assume that $v$ has at least one strong neighbor, and that $R_{\ref{r11+}}$ applies.
    
    For each neighbor of $v$, we count the amount of charge leaving $v$ towards it. Note that this amount is at most $\frac12$. 
    
    Let $X$ (resp. $Y$) be the set of weak vertices that receive $1$ (resp. $\frac12$) from $v$. Note that vertices of $X$ are next to only strong vertices around $v$. Assume that a strong neighbor $w$ of $v$ is next to a vertex from $Y$, then we have two cases: 
    \begin{itemize}
        \item Either all the neighbors of $v$ except $w$ lie in $Y$, in which case no charge transits through $w$, and the final charge of $v$ is at least $5-10\cdot\frac12=0$.
        \item Otherwise, there is another strong neighbor $w'$ of $v$ such that all the neighbors of $v$ between $w$ and $w'$ lie in $Y$. In that case, at most $\frac14$ transits through $w$ and $w'$, hence the final charge of $v$ is at least $5-2\cdot\frac 14-9\cdot \frac12=0$.
    \end{itemize}
    
    Therefore, strong neighbors of $v$ are close to only vertices of $S\cup X$. Since $v$ has odd degree and at least a strong neighbor, there must be two strong vertices $v_1,v_2$ next to each other around $v$ (recall that vertices of $X$ cannot be next to each other). Now the charge leaving $v$ towards $v_1,v_2$ is at most $\frac14$, hence the final charge of $v$ is again at least $5-2\cdot \frac 14-9\cdot\frac 12=0$, which concludes. 
\end{proof}

Since vertices of degree $6$ have an initial non-negative charge and do not give nor receive any charge, it simply remains to show that $5^-$-vertices end up with non-negative charges. To this end, we consider two cases depending on the value of $\Delta$.

\subsection{$\Delta \ge 11$}

\begin{lemma}
    Vertices of degree $3$ have a positive charge after applying the discharging rules.
\end{lemma}
\begin{proof}
    
    Let $v$ be a vertex of degree $3$ with initial charge $-3$. We consider several cases depending on how many $4^+$ faces contain $v$. In each case we show that $v$ receives at least 3, so that its final charge is non-negative. 
    
    First observe that by Theorem~\ref{thm:edge} and Lemma~\ref{lem:deg3forbif4}, all the vertices on the same face as $v$ are strong. Indeed, since weak vertices have degree at most $5$, $v$ is not adjacent to a weak vertex and $v$ is not on the same face that a non-adjacent vertex of degree less than $\Delta$. In particular, for every neighbor $u$ of $v$, $v$ is next to two strong vertices around $u$.
    
    If $v$ is adjacent to three $4^+$-faces, then $v$ receives at least $3$ from these faces by $R_{\ref{rface}}$.

     If $v$ is adjacent to two $4^+$-faces, then $v$ receives at least $2$ from these faces by $R_{\ref{rface}}$. Consider $f$ one of these faces. By Theorem \ref{thm:edge}, one of the neighbors $u$ of $v$ on $f$ has degree $\Delta$. Since $v$ is close to two strong vertices around $u$, it receives an additional $1$ from $u$ by $R_{\ref{r11+}}$, for a total of $3$.

     Otherwise, $v$ is adjacent to at most one $4^+$-face, so $v$ is adjacent to at least two $3$-faces. By Lemma~\ref{lem:deg3forbif}, every neighbor of $v$ has degree at least $11$.
    So the vertex $v$ receives a charge of $3$ from its neighbors by $R_{\ref{r11+}}$, since, for every neighbor $u$ of $v$, $v$ is next to two strong vertices around $u$.
\end{proof}

\begin{lemma}
    Vertices of degree $4$ have a positive charge after applying the discharging rules.
\end{lemma}
        \begin{figure}[hbtp]
        \begin{center}
        \tikzstyle{vertex}=[circle,draw, minimum size=15pt, scale=1, inner sep=1pt]
        \tikzstyle{give}=[circle,draw, minimum size=15pt, scale=1, inner sep=1pt, fill=black!5, very thick]
        \tikzstyle{receive}=[circle,draw, minimum size=15pt, scale=1, inner sep=1pt, fill=black!5]
        \tikzstyle{fleche}=[->,>=latex]
        \begin{tikzpicture}[scale=.5]

            \node (a0) at (0,0) [receive] {$v$};
            \node (b0) at (0,2) [vertex]{$a$};
            \node (c0) at (0,-2) [vertex]{$c$};
            \node (d0) at (2,0) [vertex]{$b$};
            \node (e0) at (-2,0) [vertex]{$d$};
            \node at (0,-3.5) {Case 2};
            \draw (b0) to (d0);
            \draw (c0) to (d0);
            \draw (b0) to (e0);
            \draw (c0) to (e0);
            \draw (a0) to (b0);
            \draw (a0) to (c0);
            \draw (a0) to (d0);
            \draw (a0) to (e0);

            \tikzset{xshift=-14cm}
            \node (a1) at (6,0) [receive] {$v$};
            \node (b1) at (6,2) [vertex]{$a$};
            \node (c1) at (6,-2) [vertex]{$c$};
            \node (d1) at (8,0) [vertex]{$b$};
            \node (e1) at (4,0) [vertex]{$d$};
            \node (f1) at (8,2) [vertex]{$x$};
            \node at (6,-3.5) {Case 1};
            \draw (c1) to (d1);
            \draw (b1) to (e1);
            \draw (c1) to (e1);
            \draw (a1) to (b1);
            \draw (a1) to (c1);
            \draw (a1) to (d1);
            \draw (a1) to (e1);
            \draw (b1) to (f1);
            \draw (d1) to (f1);
      
        \end{tikzpicture}
        \end{center}
        \caption{Configurations for the discharging of a vertex $v$ of degree $4$.}
        \label{fig:charge4}
    \end{figure}
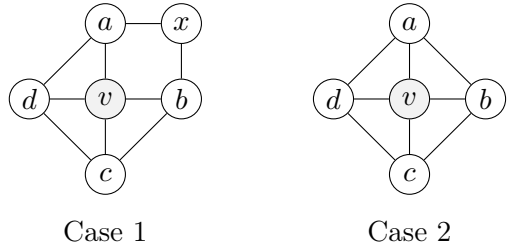
     \begin{proof}
     Let $v$ be a vertex of degree $4$. The initial charge of $v$ is $-2$. If $v$ is incident to a $5^+$-face or incident to two squares then $v$ receives at least $2$ from these faces by $R_{\ref{rface}}$. The two remaining cases are when $v$ is adjacent to four triangles or to a square and three triangles.
    \medskip

    \noindent   
\textbf{Case 1.} Assume that $v$ is incident to a square. Denote by $a, b, c, d$ the neighbors of $v$ and by $x$ the last vertex of the square as shown in Figure \ref{fig:charge4}. First note that $\deg(a)= \Delta$ or $\deg(b)= \Delta$ by Theorem \ref{thm:edge} and, similarly, $\deg(x)= \Delta$. By symmetry assume $\deg(a)= \Delta$ so $v$ receives $1$ from the square by $R_{\ref{rface}}$ and at least $\frac{1}{2}$ from $a$ by $R_{\ref{r11+}}$. 

By Lemma \ref{lem:minconstraint}, we get that $\deg(b)+\deg(c)+\deg(d)\geq \Delta+14 \geq 25$. In particular, there exists $y\in\{b,c,d\}$ with degree at least 9, and another neighbor must have degree at least 7. Therefore, $y$ is next to at least one strong vertex around $v$, so $R_{\ref{r11+}}$ or $R_{\ref{r7}}$ applies and $y$ gives either $\frac12$ or $\frac{3\omega_0(y)}{2\deg(y)}\geqslant \frac12$ depending on which rule is applied. Therefore $v$ receives at least $1+\frac12+\frac12=2$.
\medskip
    
\noindent
\textbf{Case 2.} Assume that $v$ is adjacent to four triangles. Denote by $a, b, c, d$ the neighbors of $v$ in the cyclic ordering as in Figure \ref{fig:charge4}. The inequality $\deg(a)+\deg(b)+\deg(c)+\deg(d)\geq 2 \Delta+15 \ (\star)$ holds by Lemma~\ref{lem:minconstraint}. 
    \begin{itemize}
        \item  If at least three of the four neighbors have degree at least $12$ (say, $b,c,d$), then $c$ gives $1$ to $v$ and both $b$ and $d$ give $\frac{1}{2}$ to $v$ by $R_{\ref{r11+}}$.
        \item  If $v$ has exactly two neighbors with degree at least $12$, the other two neighbors $Y$ have degree at most $11$ and the sum of their degree is at least $15$. Thus, one of the following happens:
        \begin{itemize}
            \item A vertex of $Y$ has degree $11$, $v$ has three neighbors of degree at least $11$, and $v$ receives at least $1+2\cdot\frac{1}{2}$ by $R_{\ref{r11+}}$. 
            \item The two vertices of $Y$ have degree at least $5$, and they are not weak by Lemma~\ref{lem:deg4forbif3}. Thus, $v$ receives $2\cdot 1$ from its two neighbors of degree at least $12$ by $R_{\ref{r11+}}$.
        \end{itemize}
\item  If $v$ has exactly one neighbor with degree at least $12$ (say $a$), then by $(\star)$ we have $\deg(b)+\deg(c)+\deg(d)\geq \Delta+15 \geq 27$ (note that $\Delta \geq 12$ because $\deg(a)=12$). Thus, since $b,c,d$ all have degree at most $11$, they also have degree at least $5$, hence one of the following happens:
        \begin{itemize}
            \item A vertex in $\{b,c,d\}$ has degree $5$, then the two others have degree $11$ and then $v$ receives $1+2\cdot\frac12$ by $R_{\ref{r11+}}$.
            \item Otherwise, none of $b,c,d$ are weak and at least two of them have degree at least $8$. In that case, $v$ receives at least $2\cdot \frac{1}{2}$ from them by $R_{\ref{r7}}$, and 1 from $a$ by $R_{\ref{r11+}}$.
        \end{itemize}
        \item If $v$ has no $12^+$-neighbor then $\deg(a)+\deg(b)+\deg(c)+\deg(d)\geq 2 \Delta+15 \geq 37$. 
        \begin{itemize}
            \item If $v$ has a $5^-$-neighbor (say $a$), then the three other neighbors are strong. By $(\star)$, if $\deg(a)=4$ then $b,c,d$ have degree 11 and they give $2$ to $v$ by  $R_{\ref{r11+}}$. Otherwise, $\deg(a)=5$ and by Lemma~\ref{lem:deg4forbif3}, $a$ is strong. Moreover, two of the three remaining vertices have degree $11$ so they give $2\cdot 1$ to $v$ by $R_{\ref{r11+}}$.
            
\item If the four neighbors have degree at least $6$, all of them are strong. And, because of $(\star)$, either $v$ has two $11$-neighbors giving it $2\cdot 1$ by $R_{\ref{r11+}}$, or it has three $9^+$-neighbors giving it $3\cdot\frac23$, or it has four neighbors of degree at least 8 giving it $4\cdot\frac 12$ by $R_{\ref{r7}}$.
        \end{itemize}
    \end{itemize}
\end{proof}

\begin{lemma}
    Vertices of degree $5$ have a positive charge after applying the discharging rules.
\end{lemma}
        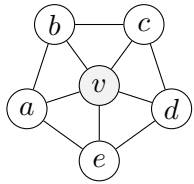
\begin{figure}[hbtp]
        \begin{center}
        \tikzstyle{vertex}=[circle,draw, minimum size=15pt, scale=1, inner sep=1pt]
        \tikzstyle{give}=[circle,draw, minimum size=15pt, scale=1, inner sep=1pt, fill=black!5, very thick]
        \tikzstyle{receive}=[circle,draw, minimum size=15pt, scale=1, inner sep=1pt, fill=black!5]
        \tikzstyle{fleche}=[->,>=latex]
        \begin{tikzpicture}[scale=1]
            
            \node (a) at (198:1) [vertex] {$a$};
            \node (b) at (126:1) [vertex] {$b$};
            \node (c) at (54:1) [vertex] {$c$};
            \node (d) at (342:1) [vertex] {$d$};
            \node (e) at (270:1) [vertex] {$e$};
            \node (f) at (0,0) [receive] {$v$};
            
            \draw (a) to (f);
            \draw (b) to (f);
            \draw (c) to (f);
            \draw (d) to (f);
            \draw (e) to (f);
            \draw (a) to (b);
            \draw (b) to (c);
            \draw (c) to (d);
            \draw (d) to (e);
            \draw (e) to (a);
        \end{tikzpicture}
        \end{center}
        \caption{Configuration around the vertex $v$ for the discharging of a $5$-vertex}
        \label{fig:charge5}
    \end{figure}
         
\begin{proof}
    Let $v$ be a vertex of degree $5$, its initial charge is $-1$. If $v$ is adjacent to a $4^+$-face then $v$ receives at least $1$ from this face by  $R_{\ref{rface}}$. The remaining case is when $v$ is adjacent to five triangles. We denote by $a,b,c,d,e$ its neighbors as in Figure \ref{fig:charge5}. By Lemma \ref{lem:minconstraint}, $\deg(a)+\deg(b)+\deg(c)+\deg(d)+\deg (e)\geq 2 \Delta+17$ $(\star\star)$. If two neighbors of $v$ have degree at least $12$ then $v$ receives at least $\frac{1}{2}$ from each by $R_{\ref{r11+}}$. 
    \medskip
    
    \noindent
    \textbf{Case 1.} $v$ has a $12^+$-neighbor, say $a$. \\ Then $(\star\star)$ yields $\deg(b)+\deg(c)+\deg(d)+\deg(e)\geq  \Delta+17 \geq 29$ so there are at most two $5^-$-vertices in $\{b,c,d,e\}$.
    \begin{itemize}
        \item  If $v$ has two $5^-$-neighbors, then its two other neighbors either have degree 8 and 11 or both have degree at least $9$. In each case, $v$ receives at least $\min(\frac14+\frac{5}{11},2\cdot \frac{1}{3})\geq \frac12$ from them by $R_{\ref{r11}}$ and $R_{\ref{r7}}$, plus at least $\frac12$ from $a$.
        \item  If $v$ has exactly one $5^-$-neighbor $w$ then, if $w$ is not adjacent to $a$, $a$ gives $1$ to $v$ by $R_{\ref{r11+}}$. Assume now that $w$ is adjacent to $a$, say $w=b$. And $v$ receives $\frac{1}{2}$ from $a$ by $R_{\ref{r11+}}$.
        By $(\star\star)$, we get $\deg(c)+\deg(d)+\deg(e)\geq 24$. Since all these vertices are strong and consecutive on $f$, they all give at least $\frac 32 \cdot \frac{\omega_0}{11}$ to $v$. Since $\omega_0(c)+\omega_0(d)+\omega_0(e)=\deg(c)+\deg(d)+\deg(e)-18\geqslant 6$, $v$ receives at least $\frac32 \cdot \frac{6}{11} > \frac 12$ from $c,d,e$. So in total $v$ receives a charge larger than one.
\item If $v$ has no $5^-$-neighbor, then each of its neighbors is strong and $v$ receives $1$ from $a$ by $R_{\ref{r11+}}$.
    \end{itemize}
   \medskip
   
   \noindent
    \textbf{Case 2.} $v$ has no $12^+$-neighbor. \\
    Recall that by $(\star\star)$, $\deg(a)+\deg(b)+\deg(c)+\deg(d)+\deg (e)\geq 39$ so $v$ has at most two $5^-$-neighbors. 
    \begin{itemize}
        \item If $v$ has two $5^-$-neighbors, then by $(\star\star)$, the three others all have degree at least $7$. One of the following happens: 
        \begin{itemize}
            \item Three vertices have degree at least $8$, then two of them are next to each other around $v$, so $v$ receives at least $2 \cdot \frac{3}{8}+\frac 14$ from them by $R_{\ref{r7}}$.
            \item Two vertices have degree $11$ and the last one has degree 7. Thus, $v$ receives at least $2 \cdot \frac{5}{11} + \frac{1}{7} > 1$ by $R_{\ref{r11}}$, $R_{\ref{r11+}}$ and $R_{\ref{r7}}$. 
        \end{itemize}
        \item If $v$ has at most one $5^-$-neighbor, then all the vertices have at least one strong neighbor on the cycle. Thus $v$ receives charge at least $\frac32 \sum_{w \in N(v)} \frac{\omega_0(w)}{\deg(w)} \ge \frac 32\cdot\frac{39-30}{11} >1$.
\end{itemize}
     In each case, $v$ gets at least $1$ charge, which concludes the proof. 
\end{proof}

\subsection{$\Delta \in \{ 9 ,10 \}$}

\begin{lemma}
    Vertices of degree $3$ have non-negative charge after applying the discharging rules.
\end{lemma}
\begin{proof}
    Let $v$ be a vertex of degree $3$ with initial charge $-3$.
    Note that $v$ does not give charge.
    The vertex $v$ cannot be adjacent to more than one $3$-face by Lemma~\ref{lem:deg3forbif}.
    
    If $v$ is adjacent to one $3$-face, the other faces incident to $v$ cannot be two $4$-faces by Lemma~\ref{lem:deg3forbif2}.
    Hence $v$ is incident to a $4^+$-face and a $5^+$-face.
    By $R_{\ref{rface}}$, $v$ receives at least $3$ from these faces.
    Otherwise, $v$ is adjacent to three $4^+$-faces and $v$ receives at least $3$ from these faces by $R_{\ref{rface}}$. 
    In both cases, $v$ has a non-negative final charge.
\end{proof}

\begin{lemma}
    Vertices of degree $4$ have non-negative charge after applying the discharging rules.
\end{lemma}

\begin{proof}
    Let $v$ be a vertex of degree $4$ with initial charge $-2$. If $v$ is incident to a $5^+$-face or to two squares then $v$ receives at least $2$ from these faces by $R_{\ref{rface}}$. The two remaining cases are when $v$ is adjacent to four triangles or $v$ is adjacent to a square and three triangles.
\medskip

    \noindent
\textbf{Case 1.} Assume that $v$ is incident to a square. Denote by $a, b, c, d$ the neighbors of $v$ and by $x$ the last vertex of the square as shown in Figure \ref{fig:charge4}. First note that $\deg(a)= \Delta$ or $\deg(b)= \Delta$ by Theorem \ref{thm:edge} and, similarly, $\deg(x)= \Delta$. By symmetry assume $\deg(a)= \Delta$ so $v$ receives $1$ from the square by $R_{\ref{rface}}$. We show that $v$ receives 1 from $a,b,c,d$.

By Lemma~\ref{lem:minconstraint}, we get that $\deg(b)+\deg(c)+\deg(d)\geq \Delta+14 \geq 23$. Therefore, either $v$ has two $8^+$-neighbors among $\{b,c,d\}$ and it receives $\frac12+2\cdot \frac{1}{4}=2$, or $b,c,d$ have degree $9,7,7$ and $v$ receives $\frac{1}{3} + \frac{1}{3} + \frac{2}{7} + \frac{2}{7} \ge 1$ by $R_{\ref{r11+}}$ and $R_{\ref{r7}}$.
\medskip

\noindent
\textbf{Case 2.} Assume that $v$ is adjacent to four triangles. Denote by $a, b,c, d$ the neighbors of $v$ as shown in Figure \ref{fig:charge4}.
By Lemma~\ref{lem:minconstraint}, we get $\deg(a)+\deg(b)+\deg(c)+\deg(d)\geq 2 \Delta+15 \ge 33 \ (\star)$ hence $v$ has at most one $7^-$-neighbor. 

If $v$ has no $10$-neighbor, then all its neighbors have degree at least $6$ so they are all strong. So all the vertices neighbors of $v$ give charge at least $\frac{2 \omega_0}{\deg}$ to $v$. So in total $v$ receives a charge of $2 \cdot \frac{33-24}{9} = 2$.

Otherwise, $v$ has a 10-neighbor, say $a$ (and then $\Delta=10$). Since $\deg(b)+\deg(c)+\deg(d)\geq 25$ (since $\Delta=10$), one of the following happens:
            \begin{itemize}
            \item  One neighbor among $\{b,c,d\}$ has degree $5$ and the two others have degree $10$ so $v$ receives charge at least $\frac{4}{5} + 2\cdot\frac{3}{5} = 2$ from $a,b,c,d$ by $R_{\ref{r7}}$.
            \item  $b,c,d$ are $6^+$-vertices and then $v$ only have strong neighbors. So by $R_{\ref{r7}}$, $v$ receives charge at least $2\cdot \frac{35-24}{10} >2$ from $a,b,c,d$.
        \end{itemize}
    \medskip
    
    In all cases, $v$ receives at least $2$ and has a non-negative final charge, which concludes.
\end{proof}

\begin{lemma}
    Vertices of degree $5$ have a positive charge after applying the discharging rules.
\end{lemma}

\begin{proof}
    Let $v$ be a vertex of degree $5$, its initial charge is $-1$. If $v$ is adjacent to a face of degree at least $4$ then $v$ receives at least $1$ charge by $R_{\ref{rface}}$. The remaining case is when $v$ is adjacent to five triangles. We denote by $a, b,c,d,e$ its neighbors as in Figure \ref{fig:charge5}. By Lemma \ref{lem:minconstraint}, $\deg(a)+\deg(b)+\deg(c)+\deg(d)+\deg (e)\geq 2 \Delta+17$ $(\star)$ so $v$ has at most two $5^-$-neighbors.
    \begin{itemize}
        \item If $v$ has exactly two $5^-$-neighbors, then its three other neighbors have a total degree of at least $2 \Delta +7$. Note that two of these neighbors are next to each other around $v$. If $v$ has a $7$-neighbor then the two others have degree $\Delta$ they all give at least $\frac{3}{14}+\frac{1}{2}+\frac{1}{3}>1$ to $v$ by $R_{\ref{r7}}$. Otherwise, $v$ has three $8^+$-neighbors, in which case $v$ similarly receives at least $2\cdot\frac 38 + \frac 14 = 1$.
        \item If $v$ has exactly one $5^-$-neighbor (say $a$), then $b,c,d,e$ are strong. By $(\star)$, either all of them have degree at least $7$ and $v$ receives at least $2\cdot\frac{2}{7} + 2\cdot\frac{3}{14} = 1$, or three of them have degree at least 8 and $v$ receives at least $\frac{1}{2} + 2\cdot\frac{3}{8} = 1$, or two of them have degree at least $9$ and $v$ receives at least $2\cdot\frac{1}{2} = 1$ by $R_{\ref{r7}}$.
\item If $v$ has no $5^-$-neighbor, then all its neighbors are strong. So each neighbor $u$ of $v$ gives $\frac{2 \omega_0(u)}{\deg(u)}$ to $v$ and then $v$ receives charge at least $2 \cdot \frac{35-30}{10}=1$.
\end{itemize}
        In all cases, $v$ receives at least $1$ hence ends up with non-negative final charge.
\end{proof}

\section{Extension to Theorem~\ref{thm:delta6}}

We present the proof of Theorem~\ref{thm:delta6}. We re-use the same analysis as before: the discharging rules are the same, but the configurations are more general. We start with a minimum counterexample to Theorem~\ref{thm:delta6}, i.e. a planar graph $G$ with $\Delta(G)\leqslant 6$ and $\chi_2(G)>21$, which is minimal for $\leq$. We first present an adapted version of Lemma~\ref{lem:minconstraint}, whose proof is similar. 

\begin{lemma}
\label{lem:aux}
    The graph $G$ does not contain a vertex $v$ such that all the pairs of neighbors of $v$ are at distance at most $2$ in $G-v$, and $v$ has at least $21$ neighbors at distance $2$ in $G$.
\end{lemma}

As a consequence, we can easily show that $G$ has no weak vertex. 

\begin{corollary}\label{lem:delta6}
    The graph $G$ does not contain a triangulated $5$-vertex, or a $4$-vertex incident to three triangles.
\end{corollary}

We also need the following configuration.
\begin{lemma}\label{lem:delta6b}
     The graph $G$ does not contain a $3$-vertex adjacent to a $3$-face.
\end{lemma}
\begin{proof}
     If $G$ contains a $3$-vertex $v$ adjacent to a $3$-face, let $x$ and $y$ be the two other vertices of the $3$-face and $z$ the last vertex adjacent to $v$. Let $G'$ be the graph obtained from $G-v$ by arbitrarily adding one of the edges $xz$ or $yz$ if they are not already present in $G$. 
     
     By minimality, $G'$ admits a distance-$2$ $(2 \Delta +7)$-coloring $\alpha$. The coloring $\alpha$ is a partial coloring of $G$ and forbids at most $16$ colors for $v$, a contradiction with Lemma~\ref{lem:aux}.
\end{proof}

We now reach a contradiction applying the rules $R_{\ref{rface}}$ to $R_{\ref{r7}}$, which concludes the proof of Theorem~\ref{thm:delta6}. 
\begin{lemma}
    All faces and vertices have non-negative charge after applying of discharging rules.
\end{lemma}

\begin{proof}
    First note that $\Delta\leq 6$ imply that only $R_{\ref{rface}}$ can happen. 
    \begin{itemize}
        \item The same proof as Claim~\ref{claim:faces} ensures that faces end up with non-negative final charge. 
    
    \item $G$ does not contain vertices of degree $1$ or $2$ by Lemmas~\ref{lem:deg1forbif} and~\ref{lem:deg2forbif}. 
    \item By Lemmas~\ref{lem:delta6} and~\ref{lem:delta6b}, each vertex $v$ of degree $d\in\{3,4,5\}$ is incident to at least $6-d$ $4^+$-faces. Thus $v$ receives at least 1 from each of them, hence its final charge is at least $d-6+6-d=0$.
    \item Vertices of degree $6$ are not affected by the rule, hence keep their initial zero charge.
    \end{itemize}
    In each case, faces and vertices of $G$ end up with non-negative charge.
\end{proof}

\bibliographystyle{abbrv}

\end{document}